\documentclass[12pt]{article}
 \usepackage{amsmath,amssymb,amsthm}
 \RequirePackage[dvips]{graphicx}

 \usepackage{cite}

 \def\draw #1 by #2 (#3){
  \vbox to #2{
    \hrule width #1 height 0pt depth 0pt
    \vfill
    \special{picture #3} 
    }
  }

 \def\scaleddraw #1 by #2 (#3 scaled #4){{
  \dimen0=#1 \dimen1=#2
  \divide\dimen0 by 1000 \multiply\dimen0 by #4
  \divide\dimen1 by 1000 \multiply\dimen1 by #4
  \draw \dimen0 by \dimen1 (#3 scaled #4)}
  }

\newtheorem{theorem}{Theorem}[section]
\newtheorem{example}{Example}
\newtheorem{problem}[example]{Problem}
\newtheorem{defin}[theorem]{Definition}
\newtheorem{lemma}[theorem]{Lemma}
\newtheorem{corollary}[theorem]{Corollary}

\newtheorem{nt}{Note}

 \newcommand{\singlespacing}{\let\CS=\@currsize\renewcommand{\baselinestretch}{1}\tiny\CS}
 \newcommand{\oneandahalfspacing}{\let\CS=\@currsize\renewcommand{\baselinestretch}{1.25}\tiny\CS}
 \newcommand{\doublespacing}{\let\CS=\@currsize\renewcommand{\baselinestretch}{1.35}\tiny\CS}

 \newtheorem{conjecture}[theorem]{Conjecture}

 \newtheorem{rule-def}[theorem]{Rule}

\begin{document}
\baselineskip 16pt
 \newcommand{\la}{\lambda}
 \newcommand{\si}{\sigma}
 \newcommand{\ol}{1-\lambda}
 \newcommand{\be}{\begin{equation}}
 \newcommand{\ee}{\end{equation}}
 \newcommand{\bea}{\begin{eqnarray}}
 \newcommand{\eea}{\end{eqnarray}}

 \baselineskip=0.30in

 \begin{center}

{\Large \bf On Zagreb indices of graphs}  \\

 \vspace{4mm}

{\large \bf Batmend Horoldagva$^a$} and {\large \bf Kinkar Chandra Das$^{b}$}

 \vspace{7mm}

 \baselineskip=0.20in

 $^a${\it Department of Mathematics, Mongolian National University of Education, \\
 Baga toiruu-14, Ulaanbaatar, Mongolia\/} \\
  {\rm e-mail:} {\tt horoldagva@msue.edu.mn}\\[2mm]

 $^b${\it Department of Mathematics, Sungkyunkwan University, \\
 Suwon 16419, Republic of Korea\/} \\
 {\rm e-mail:} {\tt kinkardas2003@googlemail.com}

 \vspace{4mm}

 \end{center}

 \vspace{5mm}

 \baselineskip=0.20in

 \begin{abstract}

 Let ${\mathcal G}_n$ be the set of class of graphs of order $n$. The first Zagreb index $M_1(G)$ is equal to the sum of squares of the
degrees of the vertices, and the second Zagreb index $M_2(G)$ is equal to the sum
of the products of the degrees of pairs of adjacent vertices of the underlying
molecular graph $G$. The three set of graphs are as follows:
 \begin{eqnarray*}
 &&A=\left\{G\in {\mathcal G}_n:\,\frac{M_1(G)}{n}>\frac{M_2(G)}{m}\right\},~B=\left\{G\in {\mathcal G}_n:\,\frac{M_1(G)}{n}=\frac{M_2(G)}{m}\right\}\\
 \mbox{ and }&&\\
 &&~~~~~~~~~~~~~~~~~~~~~~~~~C=\left\{G\in {\mathcal G}_n:\,\frac{M_1(G)}{n}<\frac{M_2(G)}{m}\right\}.
 \end{eqnarray*}
 In this paper we prove that $|A|+|B|<|C|$. Finally, we give a conjecture $|A|<|B|$.

 \bigskip

 \noindent
 {\bf AMS Classification:} 05C07, 05C35, 05C90\\
 \noindent
 {\bf Keywords:} Graph, First Zagreb index, Second Zagreb index

 \end{abstract}

 \baselineskip=0.30in

 \section{Introduction}

Let $G=(V,E)$ be a simple graph with vertex set $V(G)=\{v_1,\,v_2,\ldots,\,v_n\}$ and edge set $E(G)$\,, where $|V(G)|=n$ and $|E(G)|=m$. Let $\overline{G}$ be the complement of
$G$. We denote by $d_i=d_G(v_i)$ the degree of vertex $v_i$ for $i=1,\,2,\ldots,\,n$. Let ${\mathcal G}_n$ be the set of class of graphs of order $n$. For $S\subseteq {\mathcal G}_n$, let $|S|$ be the number of graphs in the set $S$. For any two nonadjacent vertices $v_i$ and $v_j$ in graph $G$, we use $G+v_iv_j$ to denote the graph obtained from adding a new edge $v_iv_j$ to graph $G$. Similarly, for $v_iv_j\in E(G)$, we use $G-v_iv_j$ to denote the graph obtained from deleting an edge $v_iv_j$ to graph $G$. The first Zagreb index $M_1(G)$ and the second Zagreb index $M_2(G)$ is defined as follows:
$$M_1(G)=\sum_{v_i\in V} d_i^2 \quad \text{and} \quad M_2(G)=\sum_{v_iv_j\in E(G)}d_i\,d_j.$$
The Zagreb indices $M_1$ and $M_2$ were first introduced by Gutman and Trinajstić in 1972, the
quantities of the Zagreb indices were found to occur within certain approximate expressions
for the total $\pi$-electron energy \cite{GT}. For more details of the mathematical
theory and chemical applications of the Zagreb indices, see \cite{BES, D, DG, GD, GRTW, HBDL,NKMT, PPS, SH, XDB, WY}.

\vspace*{3mm}

Let us consider the three sets $A$, $B$ and $C$ be as follows:
\begin{eqnarray*}
	&&A=\left\{G\in {\mathcal G}_n:\,\frac{M_1(G)}{n}>\frac{M_2(G)}{m}\right\},~B=\left\{G\in {\mathcal G}_n:\,\frac{M_1(G)}{n}=\frac{M_2(G)}{m}\right\}\\
\mbox{ and }\\
	&&~~~~~~~~~~~~~~~~~~~~~~~~~C=\left\{G\in {\mathcal G}_n:\,\frac{M_1(G)}{n}<\frac{M_2(G)}{m}\right\}.
\end{eqnarray*}
Thus we have $|A|+|B|+|C|=|{\mathcal G}_n|$ as $A\cap B=\emptyset$, $B\cap C=\emptyset$ and $C\cap A=\emptyset$.

\vspace*{3mm}

 Caporossi and Hansen \cite{CH} conjectured that
 $A=\emptyset$.   Although this conjecture is disproved for
 general graphs \cite{HV}, it was the beginning of a long series of studies to characterize the graphs $G$ for which  $G\in A$ or $G\in B$ or $G\in C$, see \cite{CHV, D1, DGH, FGE, HD2, HD4, HDS, HL, SHD, SM, SCh, VG, VSS} and
 the references cited therein. For a more detailed discussion of the comparison between the classical Zagreb indices we refer to the monograph \cite{Hd}.

 \vspace{3mm}

 \noindent
 In this paper, we prove that $|A|+|B|<|C|$.  Finally, we give a conjecture $|A|<|B|$.

\section{Main Result}

\noindent
In this section we compare three classes of graphs. For this we need the following results.

\begin{lemma} \label{1r1} Let $G$ be a graph of order $n>1$ and size $m$.\\
	\rm{(i)} If $G\in A$, then $\overline{G}\in C$.\\
	\rm{(ii)} If $G$ is irregular and $G\in B$, then $\overline{G}\in C$.
\end{lemma}

\begin{proof} From the results in \cite{DG,KKN}, we have
	\begin{equation}
	M_2(\overline{G})=\frac{n(n-1)^3}{2}-3m(n-1)^2+2m^2+\left(n-\frac{3}{2}\right)M_1(G)-M_2(G)\label{eq1}
	\end{equation}
	and
	\begin{equation}
	M_1(\overline{G})=n(n-1)^2-4m(n-1)+M_1(G).\label{eq2}
	\end{equation}
	On the other hand,  it is well known that
	\begin{equation}
	M_1(G)\geq\frac{4m^2}{n}\label{eq4}
	\end{equation}
	with equality if and only if $G$ is a regular graph.
	Clearly, $\mid V(\overline{G})\mid=n$ and $\mid E(\overline{G})\mid=n(n-1)/2-m$.
	Using (\ref{eq4}), from (\ref{eq1}) and (\ref{eq2}), we obtain
	\begin{eqnarray}
	|V(\overline{G})|M_2(\overline{G})-|E(\overline{G})|M_1(\overline{G})
	&=&nM_2(\overline{G})-(n(n-1)/2-m)M_1(\overline{G})\nonumber\\
	&=&(n-2)\left(\frac{n}{2}M_1(G)-2m^2\right)-nM_2(G)+mM_1(G)\nonumber\\
	&\geq&mM_1(G)-nM_2(G)\label{eq3}
	\end{eqnarray}
	with equality if and only if $G$ is  regular.\\
	\rm{(i)} If $G\in A$, then $mM_1(G)-nM_2(G)>0$. From (\ref{eq3}), we have $|V(\overline{G})|M_2(\overline{G})-|E(\overline{G})|M_1(\overline{G})>0$, that is,
    $\overline{G}\in C$.\\
	\rm{(ii)} Similarly, if $G$ is irregular and $G\in B$, then $\overline{G}\in C$ from the definition of $B$ and (\ref{eq3}).
\end{proof}

\begin{lemma} \label{1r2} Let $G$ be a regular graph of order $n> 3$. Then\\
\noindent
	\rm{(i)} $G-e\in C$, where $e=v_iv_j\in E(G)$,
	
\noindent
	\rm{(ii)} $G+e\in C$, where $e=v_iv_j\notin E(G)$.
\end{lemma}

\begin{proof} Let $r$ be the degree of the regular graph $G$. Then  $\mid E(G)\mid=nr/2$.\\
\noindent
	\rm{(i)} By the definition of the Zagreb indices, we have
	\begin{eqnarray*}
	&&M_1(G-e)=(n-2)r^2+2(r-1)^2=nr^2-4r+2\nonumber\\
\mbox{and}&&\\
	&&M_2(G-e)=2(r-1)r(r-1)+\left(\frac{nr}{2}-2r+1\right)r^2=\frac{nr^3}{2}-3r^2+2r.\nonumber
	\end{eqnarray*}
	Then from the above, we get
	\begin{eqnarray}
	nM_2(G-e)-(nr/2-1)M_1(G-e)=(n-4)r+2>0\nonumber
	\end{eqnarray}
	as $n>3$. Therefore $G-e\in C$ because $\mid E(G-e)\mid=nr/2-1$.
	
\noindent
	\rm{(ii)} For $e=v_iv_j\notin E(G)$, by the definition of the Zagreb indices, we have
	\begin{eqnarray*}
	&&M_1(G+e)=(n-2)r^2+2(r+1)^2=nr^2+4r+2\\
\mbox{ and }&&\\
	&&M_2(G+e)=2r(r+1)r+\left(\frac{nr}{2}-2r\right)r^2+(r+1)^2=\frac{nr^3}{2}+3r^2+2r+1.
	\end{eqnarray*}
	Then from the above, we get
	\begin{eqnarray}
	nM_2(G+e)-(nr/2+1)M_1(G+e)=(n-4)(r+1)+2>0\nonumber
	\end{eqnarray}
	as $n> 3$. Therefore $G+e\in C$ because $\mid E(G+e)\mid=nr/2+1$.
\end{proof}

\noindent
We now give our main result as follows:
\begin{theorem} \label{1k1} Let ${\mathcal G}_n$ be the set of class of graphs of order $n>3$. Let the three sets $A,\,B,\,C\subseteq {\mathcal G}_n$ be defined before. Then $|A|+|B|<|C|$.
\end{theorem}

\begin{proof} First we assume that $G$ is an irregular graph. If $G\in A\cup B$, then by Lemma \ref{1r1}, $\overline{G}\in C$. Next we assume that $G$ is a regular graph. Then by Lemma \ref{1r2}, we obtain $G-e\in C$ $(e\in E(G))$ and $G+e\in C$ $(e\notin E(G))$. Thus we conclude that if any graph $G$ in $A\cup B$ then there exists a graph $H$ $(\cong\overline{G}\mbox{ or }G-e\mbox{ or }G+e)$ in $C$, that is, $G\in A\cup B$ implies that $H\in C$.

\vspace*{3mm}

Let $G_1$ and $G_2$ $(G_1\ncong G_2)$ be any two graphs in $A\cup B$. Again let $H_1$ and $H_2$ be the graphs in $C$ such that $G_1$ corresponds to $H_1$ and $G_2$  corresponds to $H_2$. We have to prove that $H_1$ and $H_2$ are not isomorphic. When $G_1$ and $G_2$ are both irregular, then by Lemma \ref{1r1}, we obtain
             $$H_1\cong \overline{G_1}\ncong \overline{G_2}\cong H_2.$$
When $G_1$ and $G_2$ are both regular, then by Lemma \ref{1r2}, $H_1$ and $H_2$ are not isomorphic. Otherwise, one of them ($G_1$ or $G_2$) is regular and the other one is irregular. Without loss of generality, we can assume that $G_1$ is regular and $G_2$ is irregular. Then $H_1\cong G_1-e$ for some $e\in E(G_1)$ and $H_2\cong \overline{G_2}$. On the contrary, suppose that $H_1$ and $H_2$ are  isomorphic.
	Then $\overline{G_2}\cong G_1-e$ and it follows that $$G_2\cong \overline{G_1-e}\cong \overline{G_1}+e.$$
	Therefore by Lemma \ref{1r2} (ii), we have $G_2\in C$ since $\overline{G_1}$ is regular. This contradicts the fact that $G_2\in A\cup B$. Therefore $H_1$ and $H_2$ are not isomorphic. Hence we conclude that $|A|+|B|\leq |C|$.

\vspace*{3mm}

We now prove that the inequality is strict. For this let $H\cong K_n-e$ $(e\mbox{ is an edge in }K_n)$, $n\geq 3$. Then $\overline{H}\cong K_2\cup (n-2)\,K_1$. Thus we have
$$M_1(H)=(n-2)(n-1)^2+2(n-2)^2,~~M_2(H)=\frac{n(n-1)^3}{2}-(n-1)(3n-5),$$
and 
$$M_1(\overline{H})=2,~~M_2(\overline{H})=1.$$
One can easily check that
$$\frac{M_1(H)}{n}<\frac{M_2(H)}{m}~~\mbox{ and }~~\frac{M_1(\overline{H})}{n}<\frac{M_2(\overline{H})}{\frac{n(n-1)}{2}-m}.$$
Hence $H,\,\overline{H}\in C$. If there is no graph in $A\cup B$ correspondence to $H$ in $C$, then we have $|A|+|B|<|C|$. Otherwise, there is a graph $G$ in $A\cup B$ corresponds to $H$ in $C$. Then by Lemma \ref{1r1}, we have $\overline{G}\cong H$, that is, $G\cong \overline{H}\in C$, a contradiction as $G\in A\cup B$.
This completes the proof.
\end{proof}

\begin{corollary} \label{1k2} Let ${\mathcal G}_n$ be the set of class of graphs of order $n>3$. Also let the three sets $A,\,B,\,C\subseteq {\mathcal G}_n$ be defined before. Then $|A|<|C|$ and $|B|<|C|$.
\end{corollary}

\begin{corollary} Let ${\mathcal G}_n$ be the set of class of graphs of order $n>3$. Also let $C$ be the set defined before. Then $|C|>\frac{|{\mathcal G}_n|}{2}$.
\end{corollary}

\begin{proof} From the definitions of $A,\,B$ and $C$, we have $|A|+|B|+|C|= |{\mathcal G}_n|$. By Theorem \ref{1k1} with the above result, we obtain
  $$2|C|>|{\mathcal G}_n|,~\mbox{ that is, }|C|>\frac{|{\mathcal G}_n|}{2}.$$
\end{proof}

\noindent
Now we would like to end this paper with the following relevant conjecture.
\begin{conjecture} \label{1pk1} Let $A$ and $B$ be the two sets defined before. Then $|A|<|B|$.
\end{conjecture}


\end{document}